\documentclass{my-amsart}

\usepackage{amsmath,amssymb,amsthm}
\usepackage{url}

\newtheorem{theorem}{Theorem}[section]
\newtheorem{lemma}[theorem]{Lemma}
\newtheorem{corollary}[theorem]{Corollary}
\newtheorem{proposition}[theorem]{Proposition}

\theoremstyle{definition}
\newtheorem{definition}[theorem]{Definition}

\theoremstyle{remark}
\newtheorem{remark}[theorem]{Remark}

\numberwithin{equation}{section}

\begin{document}

\title[The Duality Theory of Fractional Calculus]{%
The Duality Theory of Fractional Calculus\\ 
and a New Fractional Calculus of Variations Involving Left Operators Only}

\author[D. F. M. Torres]{Delfim F. M. Torres\\
(ORCID: 0000-0001-8641-2505)}

\address{Center for Research and Development in Mathematics and Applications (CIDMA), 
Department of Mathematics, University of Aveiro, 3810-193 Aveiro, Portugal}

\email{delfim@ua.pt}

\subjclass[2020]{34A08, 49K05, 49S05}

\date{Please cite this paper as follows:
D. F. M. Torres, The duality theory of fractional calculus 
and a new fractional calculus of variations involving left operators only,
Mediterr. J. Math. 21 (2024), no.~3, Paper No.~106, 16~pp.
(\url{https://doi.org/10.1007/s00009-024-02652-x}).}

% -------------------------------

\keywords{Duality, fractional calculus,
integration by parts, fractional calculus of variations,
Euler--Lagrange equations, dissipative systems.}

\begin{abstract}
Through duality it is possible to transform left 
fractional operators into right fractional operators and vice versa. 
In contrast to existing literature, we establish integration 
by parts formulas that exclusively involve either left or right operators. 
The emergence of these novel fractional integration by parts formulas inspires 
the introduction of a new calculus of variations, where only one type of 
fractional derivative (left or right) is present. This applies to both the 
problem formulation and the corresponding necessary optimality conditions.
As a practical application, we present a new Lagrangian that relies solely 
on left-hand side fractional derivatives. The fractional variational principle 
derived from this Lagrangian leads us to the equation of motion 
for a dissipative/damped system.
\end{abstract}

\maketitle

% -------------------------------

\section{Introduction}

Fractional differentiation means ``differentiation of arbitrary order''.
Its origin goes back more than 325 years, when in 1695 L'Hopital asked 
Leibniz the meaning of $\frac{d^{n}x}{dt^{n}}$ for $n=1/2$.
After that, many famous mathematicians, like Fourier,
Abel, Liouville, and Riemann, among others, contributed to
the development of Fractional Calculus \cite{MR1347689}.

In 1931, Bauer proved that it is impossible to use a
variational principle to derive a linear dissipative
equation of motion with constant coefficients \cite{Bauer:1931}.
Bauer's result expresses the well-known belief that there is no direct
method of applying variational principles to nonconservative
systems, which are characterized by friction or other dissipative
processes. It turns out that fractional derivatives provide 
an elegant solution to the problem.
Indeed, the proof of Bauer relies, implicitly, 
on the assumption that derivatives are of integer order.

The Fractional Calculus of Variations (FCoV) was born in 1996--1997
with the work of Riewe \cite{CD:Riewe:1996,CD:Riewe:1997}, 
precisely with the aim to obtain a Lagrangian
for a simple dissipative system with a damping force
proportional to the velocity. Riewe proposed to represent 
the dissipative effects with a Lagrangian dependent
on left and right fractional Riemann--Liouville derivatives 
of order $1/2$, showing that in such case one can obtain the equations 
of motion of a dissipative linear system with constant coefficients 
by a (fractional) variational principle.
Riewe's idea is very simple and natural:
if the Lagrangian contains a term proportional 
to $\left(\frac{d^{n}x}{dt^{n}}\right)^2$, then the respective 
Euler--Lagrange equation has a corresponding 
term proportional to $\frac{d^{2n}x}{dt^{2n}}$. Therefore,
a damping force of the form $c \frac{dx}{dt}$ should
follow from a Lagrangian containing a term
proportional to $\left(\frac{d^{1/2}x}{dt^{1/2}}\right)^2$. 
The FCoV has therefore significant importance 
in physics and engineering, as a means 
to circumvent Bauer's result \cite{MR4085516}.
This explains the fact why FCoV is under strong development. 
For those interested on the subject, we refer to the 
books \cite{MR3443073,MR3822307,MR3331286,MR2984893}
and the survey papers \cite{MR3888408,MR3221831}.

The Lagrangian proposed by Riewe marked the beginning of several discussions 
in the literature, in particular among physicists, due to the presence 
of right fractional derivatives and other related issues. 
For a good and recent account about the fractional Lagrangian 
proposed by Riewe and other proposals by different mathematicians, 
we refer the reader to \cite{MR3390552}. Here we remark
that, to the best of our knowledge, all such fractional Lagrangians 
involve always a right derivative, which
causes some physical concerns of non-causality. 
From a strictly mathematical point of view, 
however, the right operators appear naturally in the FCoV due 
to the central role of integration by parts in the proof of the 
Euler--Lagrange necessary optimality conditions.
Indeed, under fractional integration by parts, a left fractional operator
is transformed into a right fractional operator and vice versa 
(see Section~\ref{sec:02}), so even if
the Lagrangian only includes left operators, then the right
operators will appear in the Euler--Lagrange equation.
To circumvent the undesirable phenomenon of appearance 
of right-fractional operators either on the Lagrangian or the
Euler--Lagrange equation, here we develop the theory of duality 
for fractional calculus (see Section~\ref{sec:3.1}), 
as formulated by Caputo and Torres in 2015 \cite{MyID:307}, 
with early contributions discussed in \cite{MR1347689}.

The duality theory gives a way to express 
left fractional operators in terms of right fractional
operators and the other way around. In simple terms, the duality 
theory shows that the right fractional operators 
of a function are the dual of the left operators of the dual function or, 
equivalently, the left fractional derivative/integral of a
function is the dual of the right fractional derivative/integral 
of the dual function. Here we further develop and use such duality 
theory to derive new fractional integration by parts formulas
involving only left (or right) fractional operators (see Section~\ref{sec:3.2}). 
With the help of the new fractional integration by parts formulas, 
as well as results relating fractional derivatives
with classical ones (see Section~\ref{sec:fromFDtoCD}),
we provide a new perspective to the FCoV by proposing 
a new FCoV where both Lagrangians and respective 
Euler--Lagrange equations involve only one type
of operators, e.g., both variational problem and respective necessary
optimality conditions with left-hand-side operators only
(see Section~\ref{sec:newFCoV}). As an application of the obtained results, 
we go back to the original motivation behind the seminal papers of Riewe 
and provide a new Lagrangian that circumvents the Bauer theorem, 
for which the respective Euler--Lagrange equation 
coincide with the equation of motion of a dissipative system
(see Section~\ref{sec:ex}). In contrast with the example of Riewe, 
and others available in the literature, our quadratic Lagrangian 
for the linear friction problem is a real valued Lagrangian involving 
let-hand side fractional derivatives only. This provides 
a better physical meaning to the proposed Lagrangian, in contrast 
with the ones available in the literature. We end up with
a conclusion, summarizing the main contributions of the paper 
(see Section~\ref{sec:conc}).

% -------------------------------

\section{Preliminaries}
\label{sec:02}

We recall some well known definitions and results
from fractional calculus, fixing also our notations. 
Here we only give the notions and results
that will be useful in the sequel.
For more details we refer to the book \cite{MR2984893},
where all such definitions and results are found,
and references therein.

\begin{definition}[Fractional integrals of order $\alpha$]
\label{def:2.1}
Let $\varphi \in L^1([a, b],\mathbb{R})$. The integrals
\begin{equation}
\label{eq:01}
\left(\mathcal{I}^{\alpha}_{a+} \varphi\right)(t)
= \frac{1}{\Gamma(\alpha)} 
\int_{a}^{t} \varphi(\tau) (t-\tau)^{\alpha-1} d\tau,
\quad t > a,
\end{equation}
and
\begin{equation}
\label{eq:02}
\left(\mathcal{I}^{\alpha}_{b-} \varphi\right)(t)
= \frac{1}{\Gamma(\alpha)} 
\int_{t}^{b} \varphi(\tau) (\tau-t)^{\alpha-1} d\tau,
\quad t < b,
\end{equation}
where $\alpha > 0$ and $\Gamma(\cdot)$ is the Gamma function,
are called, respectively, the left and the right fractional integrals 
of order $\alpha$. Additionally, we define
$$
\mathcal{I}^{0}_{a+} \varphi
= \mathcal{I}^{0}_{b-} \varphi
= \varphi,
$$
that is, for $\alpha = 0$ \eqref{eq:01} and \eqref{eq:02} are the identity operator. 
\end{definition}

\begin{proposition}[Integration by parts for fractional integrals]
\label{prop:2.1}
Let $\alpha > 0$ and $1/p + 1/q \leq 1 + \alpha$, $p \geq 1$, $q \geq 1$, 
with $p \neq 1$ and $q \neq 1$ in the case $1/p + 1/q = 1 + \alpha$. 
If $\varphi \in L^p([a, b], \mathbb{R})$ and $\psi \in L^q([a, b], \mathbb{R})$,
then the following equality holds:
\begin{equation}
\label{eq:11}
\int_a^b \varphi(\tau) \left(\mathcal{I}^{\alpha}_{a+} \psi\right)(\tau) d\tau
= \int_a^b \left(\mathcal{I}^{\alpha}_{b-} \varphi\right)(\tau) \psi(\tau)d\tau.
\end{equation}
\end{proposition}

\begin{proposition}
\label{prop:2.2}
If $f \in L^p([a, b], \mathbb{R})$, 
$1 \leq p \leq \infty$, $\alpha > 0$ and $\beta > 0$, then
\begin{equation}
\label{eq:03}
\left(\mathcal{I}^{\alpha}_{a+} \left(\mathcal{I}^{\beta}_{a+} f\right)\right)(t)
= \left(\mathcal{I}^{\alpha+\beta}_{a+} f\right)(t)
\end{equation}
almost everywhere on $[a, b]$.
\end{proposition}

\begin{definition}[Left/right Riemann--Liouville fractional derivatives of order $\alpha$]
\label{def:2.2}
Let $0 < \alpha < 1$ and $\varphi \in W^{1,1}([a, b], \mathbb{R})$. 
The left Riemann--Liouville derivative of order $\alpha$ is defined by
\begin{equation}
\label{eq:05a}
\left(\mathcal{D}^{\alpha}_{a+} \varphi\right)(t)
= \frac{1}{\Gamma(1-\alpha)} \frac{d}{dt}
\int_{a}^{t} \varphi(\tau) (t-\tau)^{-\alpha} d\tau,
\end{equation}
while the right Riemann--Liouville derivative of order
$\alpha$ is defined by
\begin{equation}
\label{eq:06a}
\left(\mathcal{D}^{\alpha}_{b-} \varphi\right)(t)
= -\frac{1}{\Gamma(1-\alpha)} \frac{d}{dt}
\int_{t}^{b} \varphi(\tau) (\tau-t)^{-\alpha} d\tau.
\end{equation}
For $\alpha = 0$ we have $\mathcal{D}^{0}_{a+} \varphi 
= \mathcal{D}^{0}_{b-} \varphi = \varphi$; 
while for $\alpha = 1$ one has 
$\left(\mathcal{D}^{1}_{a+} \varphi\right)(t) 
= \frac{d}{dt} \varphi(t)$
and $\left(\mathcal{D}^{1}_{b-} \varphi\right)(t) 
= -\frac{d}{dt} \varphi(t)$.
\end{definition}

\begin{remark}
\label{rem:2.6}
Let $\alpha \in [0, 1]$. For functions  
$\varphi \in W^{1,1}([a, b], \mathbb{R})$,
each of the following expressions hold:
\begin{equation}
\label{eq:05b}
\left(\mathcal{D}^{\alpha}_{a+} \varphi\right)(t)
= \frac{d}{dt} \left[\left(\mathcal{I}^{1-\alpha}_{a+} \varphi\right)(t)\right] 
\end{equation}
and 
\begin{equation}
\label{eq:06b}
\left(\mathcal{D}^{\alpha}_{b-} \varphi\right)(t)
= -\frac{d}{dt} \left[\left(\mathcal{I}^{1-\alpha}_{b-} \varphi\right)(t)\right]. 
\end{equation}
\end{remark}

The following properties can be found, for example, 
in Propositions 2.4 and 2.5 of \cite{MR2984893}.
They are valid under appropriate space of functions 
$\mathcal{I}^{\alpha}_{a+} \left(L^p(a,b)\right)$
and $\mathcal{I}^{\alpha}_{b-} \left(L^p(a,b)\right)$, 
$1 \leq p \leq \infty$, defined respectively by
$$
\mathcal{I}^{\alpha}_{a+} \left(L^p(a,b)\right)
:= \left\{\varphi : \varphi(t) = \left(\mathcal{I}^{\alpha}_{a+} \psi\right)(t), 
\psi \in L^p([a, b],\mathbb{R})\right\}
$$
and
$$
\mathcal{I}^{\alpha}_{b-} \left(L^p(a,b)\right)
:= \left\{\varphi : \varphi(t) = \left(\mathcal{I}^{\alpha}_{b-} \psi\right)(t), 
\psi \in L^p([a, b],\mathbb{R})\right\}.
$$

\begin{proposition}
\label{prop:2.4}
If $\alpha > 0$, then 
\begin{equation}
\label{eq:07}
\left(\mathcal{D}^{\alpha}_{a+} \left(\mathcal{I}^{\alpha}_{a+} \varphi\right)\right)(t)
= \varphi(t)
\end{equation}
for any $\varphi \in L^1([a, b],\mathbb{R})$, while 
\begin{equation}
\label{eq:08}
\left(\mathcal{I}^{\alpha}_{a+} \left(\mathcal{D}^{\alpha}_{a+} \varphi\right)\right)(t)
= \varphi(t)
\end{equation}
is satisfied for $\varphi \in 
\mathcal{I}^{\alpha}_{a+} \left(L^1(a,b)\right)$.
\end{proposition}

\begin{proposition}[Integration by parts for fractional derivatives]
\label{prop:2.5}
Let $0 < \alpha < 1$. If $\varphi \in \mathcal{I}^{\alpha}_{b-} \left(L^p(a,b)\right)$
and $\psi \in \mathcal{I}^{\alpha}_{a+} \left(L^q(a,b)\right)$ 
with $1/p + 1/q \leq 1 + \alpha$, then the following equality holds:
\begin{equation}
\label{eq:12}
\int_a^b \varphi(\tau) \left(\mathcal{D}^{\alpha}_{a+} \psi\right)(\tau) d\tau
= \int_a^b \left(\mathcal{D}^{\alpha}_{b-} \varphi\right)(\tau) \psi(\tau)d\tau.
\end{equation}
\end{proposition}

% -------------------------------

\section{Main Results}
\label{sec:03}

In what follows the $\alpha$'s are always real numbers between $0 < \alpha < 1$.

% -----------

\subsection{A duality theory for the FCoV}
\label{sec:3.1}

Here we extend the duality of \cite{MyID:307} to the context 
of the calculus of variations.

\begin{definition}[Dual function and the dual operator $*$]
\label{def:01}
Given $f : [a,b] \rightarrow \mathbb{R}$, we define the dual function
$f^* : [a,b] \rightarrow \mathbb{R}$ (and the dual operator $*$) by
\begin{equation}
\label{eq:13}
f^*(t) = f(b-t+a)
\end{equation}
for all $t \in [a,b]$.
\end{definition}

The following properties are straightforward but important.

\begin{proposition}
\label{prop:01} 
For any function $f : [a,b] \rightarrow \mathbb{R}$, let
$f^{**} := \left(f^{*}\right)^*$. Then,
\begin{equation}
\label{eq:13b}
f^{**} = f.
\end{equation}
If $c_1, c_2 \in \mathbb{R}$ and 
$f_1, f_2 : [a,b] \rightarrow \mathbb{R}$, then
\begin{equation}
\label{eq:13c}
\left(f_1 \cdot f_2\right)^* = f_1^* \cdot f_2^*;
\end{equation}
\begin{equation}
\label{prop:02}
\left(c_1 f_1 + c_2 f_2\right)^* = c_1 f_1^* + c_2 f_2^*.
\end{equation}
\end{proposition}

\begin{proof}
The proof is a trivial exercise.
\end{proof}

\begin{remark}
\label{cor:04}
An immediate consequence of the linearity given by \eqref{prop:02}
is that the sign change of the dual is the dual of the sign change.
This property will be useful for us to establish a relation between 
left and right fractional derivatives. Precisely,
for any function $f : [a,b] \rightarrow \mathbb{R}$ we have
\begin{equation}
\label{eq:17b}
\left(-f\right)^* = -f^*.
\end{equation}
\end{remark}

\begin{proposition}
\label{prop:04}
Under integration, the following relation holds:
\begin{equation}
\label{eq:prop:04}
\int_{a}^{b} f^*(t) \cdot g(t) dt
= \int_{a}^{b} f(t) \cdot g^*(t) dt.
\end{equation}
\end{proposition}

\begin{proof}	
From Definition~\ref{def:01} of duality one has
\begin{equation*}
\int_{a}^{b} f^*(t) \cdot g(t) dt
= \int_{a}^{b} f(b-t+a) \cdot g(t) dt
\end{equation*}
and doing the change of variables $s = b-t+a$ we have
\begin{equation*}
\int_{a}^{b} f^*(t) \cdot g(t) dt
= - \int_{b}^{a} f(s) \cdot g(b-s+a) ds
= \int_{a}^{b} f(s) \cdot g^*(s) ds.
\end{equation*}
The proof is complete.
\end{proof}

The following important lemma asserts that the left fractional integral
of the dual is the dual of the right fractional integral (of the same order).
See equation (2.19) in the classical book of Samko--Kilbas--Marichev 
\cite{MR1347689} and also \cite{MR4550832}.

\begin{lemma}
\label{lemma:01}
Let $\varphi \in L^1([a, b],\mathbb{R})$. The following relation holds:
\begin{equation}
\label{eq:14}
\mathcal{I}^{\alpha}_{a+} f^* = \left(\mathcal{I}^{\alpha}_{b-} f\right)^*.
\end{equation} 
\end{lemma}

As a corollary of Lemma~\ref{lemma:01}, it follows that
the right fractional integral of the dual is the dual 
of the left fractional integral.

\begin{corollary}
\label{cor:01}
If $f \in L^1([a, b],\mathbb{R})$, then
\begin{equation}
\label{eq:15}
\mathcal{I}^{\alpha}_{b-} f^* = \left(\mathcal{I}^{\alpha}_{a+} f\right)^*.
\end{equation} 
\end{corollary}

An equivalent way to look to the duality of the fractional integral operators
consists to say that the left fractional integral is the dual of the right
fractional integral of the dual; and the right fractional integral is the dual 
of the left fractional integral of the dual. 

\begin{corollary}
\label{cor:02}
If $\varphi \in L^1([a, b],\mathbb{R})$, then
\begin{equation}
\label{eq:15b}
\mathcal{I}^{\alpha}_{a+} f
= \left(\mathcal{I}^{\alpha}_{b-} f^*\right)^* 
\end{equation} 
and
\begin{equation}
\label{eq:14b}
\mathcal{I}^{\alpha}_{b-} f
= \left(\mathcal{I}^{\alpha}_{a+} f^*\right)^*.
\end{equation} 
\end{corollary}

Roughly speaking, now we show that, up to a minus signal, the duality relations
proved for the fractional integrals also hold true for fractional derivatives.

Lemma~\ref{lemma:02} asserts that the left fractional derivative of the dual
is the sign change of the dual of the right fractional derivative (of the same order).

\begin{lemma}
\label{lemma:02}
Let $f \in W^{1,1}([a, b], \mathbb{R})$. The following relation holds:
\begin{equation}
\label{eq:16}
\mathcal{D}^{\alpha}_{a+} f^*
= -\left(\mathcal{D}^{\alpha}_{b-} f\right)^*.
 \end{equation}
\end{lemma}

\begin{proof}
From relations \eqref{eq:05b} and \eqref{eq:14}, we can write that
\begin{equation}
\label{eq:proof:lemma:02}
\left(\mathcal{D}^{\alpha}_{a+} f^*\right)(t)
= \frac{d}{dt} \left[\left(\mathcal{I}^{1-\alpha}_{a+} f^*\right)(t)\right]
=  \frac{d}{dt} \left[\left(\mathcal{I}^{1-\alpha}_{b-} f\right)^*(t)\right].
\end{equation}
Using now \eqref{eq:13} and relation
\eqref{eq:06b} between the right fractional derivative of order $\alpha$
and the right fractional integral of order $1-\alpha$, we obtain from 
\eqref{eq:proof:lemma:02} that
\begin{equation*}
\begin{split}
\left(\mathcal{D}^{\alpha}_{a+} f^*\right)(t)
&= \frac{d}{dt} \left[\left(\mathcal{I}^{1-\alpha}_{b-} f\right)(b-t+a)\right]\\
&= -\left(\mathcal{D}^{\alpha}_{b-} f\right)(b-t+a)
= -\left(\mathcal{D}^{\alpha}_{b-} f\right)^*(t),
\end{split}
\end{equation*}
which completes the proof.
\end{proof}

As a corollary, we can also say that the right fractional derivative
of the dual is the sign change of the dual of the left fractional derivative.

\begin{corollary}
\label{cor:03}
If $f \in W^{1,1}([a, b],\mathbb{R})$, then 
\begin{equation}
\label{eq:17}
\mathcal{D}^{\alpha}_{b-} f^*
= - \left(\mathcal{D}^{\alpha}_{a+} f\right)^*.
\end{equation}
\end{corollary}

\begin{proof}
Let $g = f^*$, which by Proposition~\ref{prop:01} is equivalent
to $f = g^*$. Using Lemma~\ref{lemma:02} for $g$, we can write
from \eqref{eq:16} that 
\begin{equation*}
\left(\mathcal{D}^{\alpha}_{b-} g\right)^*(t)
= -\left(\mathcal{D}^{\alpha}_{a+} g^*\right)(t)
\end{equation*}
for all $t \in [a,b]$, that is,
\begin{equation}
\label{eq:++}
\left(\mathcal{D}^{\alpha}_{b-} f^*\right)^*(t)
= -\left(\mathcal{D}^{\alpha}_{a+} f\right)(t),
\quad t \in [a,b].
\end{equation}
Applying the dual operator $*$ both sides of \eqref{eq:++} gives,
by Proposition~\ref{prop:01} and \eqref{eq:17b}, that
\begin{equation*}
\left(\mathcal{D}^{\alpha}_{b-} f^*\right)(t)
= -\left(\mathcal{D}^{\alpha}_{a+} f\right)^*(t)
\end{equation*}
and the result is proved.
\end{proof}

Recalling Remark~\ref{cor:04}, together with the linearity of the fractional
operators, allow us to look to Lemma~\ref{lemma:02} as saying that the
left fractional derivative of the dual is the dual of the right
fractional derivative of the sign change or, in other words,
the right fractional derivative is the sign change of the dual of the 
left fractional derivative of the dual. 

\begin{corollary}
\label{cor:05}
Let $f \in W^{1,1}([a, b],\mathbb{R})$. The following relations hold:
\begin{equation}
\label{eq:16b}
\mathcal{D}^{\alpha}_{a+} f^*
= \left(\mathcal{D}^{\alpha}_{b-} (-f)\right)^*
\end{equation}
and
\begin{equation}
\label{eq:16c}
\mathcal{D}^{\alpha}_{b-} f
= - \left(\mathcal{D}^{\alpha}_{a+} f^*\right)^*.
\end{equation}
\end{corollary}

\begin{proof}
Using \eqref{eq:16} and \eqref{eq:17b} we can write that
\begin{equation*}
\mathcal{D}^{\alpha}_{a+} f^*
= -\left(\mathcal{D}^{\alpha}_{b-} f\right)^*
= \left(-\mathcal{D}^{\alpha}_{b-} f\right)^*
\end{equation*}
and equality \eqref{eq:16b} follows by the linearity 
of operator $\mathcal{D}^{\alpha}_{b-}$.
Applying the dual operator $*$ to both sides 
of \eqref{eq:16b}, already proved, gives
\begin{equation*}
\mathcal{D}^{\alpha}_{b-} (-f)
=\left(\mathcal{D}^{\alpha}_{a+} f^*\right)^*,
\end{equation*}
which is equivalent, by the linearity of 
$\mathcal{D}^{\alpha}_{b-}$, to
\begin{equation*}
-\mathcal{D}^{\alpha}_{b-} f
=\left(\mathcal{D}^{\alpha}_{a+} f^*\right)^*.
\end{equation*}
The result is proved.
\end{proof}

Similarly, Remark~\ref{cor:04} together with the linearity
of the fractional derivatives give us a new look to the duality
of Corollary~\ref{cor:03}: the right fractional derivative 
of the dual is the dual of the left fractional derivative
of the sign change or, equivalently, the left fractional
derivative is the sign change of the dual of the right
fractional derivative of the dual.

\begin{corollary}
\label{cor:06}
If $f \in W^{1,1}([a, b],\mathbb{R})$, then 
\begin{equation}
\label{eq:16d}
\mathcal{D}^{\alpha}_{b-} f^*
= \left(\mathcal{D}^{\alpha}_{a+} (-f)\right)^*
\end{equation}
and
\begin{equation}
\label{eq:16e}
\mathcal{D}^{\alpha}_{a+} f
= - \left(\mathcal{D}^{\alpha}_{b-} f^*\right)^*.
\end{equation}
\end{corollary}

\begin{proof}
We prove Corollary~\ref{cor:06} as a corollary
of Corollary~\ref{cor:05}. Let $g = f^*$, that is,
$f = g^*$. From relation \eqref{eq:16c} applied 
to $g$ we know that
\begin{equation*}
\mathcal{D}^{\alpha}_{b-} g
= - \left(\mathcal{D}^{\alpha}_{a+} g^*\right)^*
\Leftrightarrow
\mathcal{D}^{\alpha}_{b-} f^*
= - \left(\mathcal{D}^{\alpha}_{a+} f\right)^*,
\end{equation*}
which, by \eqref{eq:17b} and the linearity
of $\mathcal{D}^{\alpha}_{a+}$, is equivalent to
\begin{equation*}
\mathcal{D}^{\alpha}_{b-} f^*
= \left(- \mathcal{D}^{\alpha}_{a+} f\right)^*
\Leftrightarrow
\mathcal{D}^{\alpha}_{b-} f^*
= \left(\mathcal{D}^{\alpha}_{a+} (-f)\right)^*.
\end{equation*}
We have just proved \eqref{eq:16d}. From \eqref{eq:16b}
applied to $g$, we also know that
\begin{equation*}
\mathcal{D}^{\alpha}_{a+} g^*
= \left(\mathcal{D}^{\alpha}_{b-} (-g)\right)^*
\Leftrightarrow
\mathcal{D}^{\alpha}_{a+} f
= \left(\mathcal{D}^{\alpha}_{b-} (-f^*)\right)^*,
\end{equation*}
which by the linearity of the operator $\mathcal{D}^{\alpha}_{b-}$
and \eqref{eq:17b} is equivalent to 
\begin{equation*}
\mathcal{D}^{\alpha}_{a+} f
= \left(-\mathcal{D}^{\alpha}_{b-} f^*\right)^*
\Leftrightarrow
\mathcal{D}^{\alpha}_{a+} f
= -\left(\mathcal{D}^{\alpha}_{b-} f^*\right)^*.
\end{equation*}
The result is proved.	
\end{proof}

% -----------

\subsection{New fractional formulas of integration by parts}
\label{sec:3.2}

Now we are in a good position to prove some important results: formulas
of integration by parts, for both fractional integrals and fractional
derivatives, that involve only left operators or only right operators.
This contrasts with the results available in the literature
(cf. Propositions~\ref{prop:2.1} and \ref{prop:2.5}) for which a 
left operator on the left-hand side is converted into a right operator
on the right-hand side --- see formulas \eqref{eq:11} and \eqref{eq:12}.

As we shall see in Section~\ref{sec:newFCoV}, our new results allow us 
to develop a new fractional calculus of variations (FCoV) more suitable
for mechanics.

\begin{theorem}[Integration by parts involving left fractional integrals only]
\label{thm:01}
Let $1/p + 1/q \leq 1 + \alpha$, $p \geq 1$, $q \geq 1$, 
with $p \neq 1$ and $q \neq 1$ in the case $1/p + 1/q = 1 + \alpha$. 
If $f \in L^p([a, b], \mathbb{R})$ and $g \in L^q([a, b], \mathbb{R})$, 
then the following equality holds:
\begin{equation}
\label{eq:18}
\int_a^b f(\tau) \cdot \left(\mathcal{I}^{\alpha}_{a+} g\right)^*(\tau) d\tau
= \int_a^b \left(\mathcal{I}^{\alpha}_{a+} f\right)(\tau) \cdot g^*(\tau)d\tau.
\end{equation}
\end{theorem}

\begin{proof}
From \eqref{eq:15} and \eqref{eq:11} we can write that
\begin{equation*}
\begin{split}
\int_a^b f(\tau) \left(\mathcal{I}^{\alpha}_{a+} g\right)^*(\tau) d\tau
&= \int_a^b f(\tau) \cdot \left(\mathcal{I}^{\alpha}_{b-} g^*\right)(\tau) d\tau\\
&= \int_a^b \left(\mathcal{I}^{\alpha}_{a+} f\right)(\tau) \cdot g^*(\tau) d\tau,
\end{split}
\end{equation*}
which proves the intended relation.
\end{proof}

Similarly to Theorem~\ref{thm:01}, one can
prove an integration by parts formula involving
right fractional integrals only. 

\begin{theorem}[Integration by parts involving right fractional integrals only]
\label{thm:02}
Let $1/p + 1/q \leq 1 + \alpha$, $p \geq 1$, $q \geq 1$, 
with $p \neq 1$ and $q \neq 1$ in the case $1/p + 1/q = 1 + \alpha$. 
If $f \in L^p([a, b], \mathbb{R})$ and $g^* \in L^q([a, b], \mathbb{R})$, 
then the following equality holds:
\begin{equation}
\label{eq:19}
\int_a^b f(\tau) \cdot \left(\mathcal{I}^{\alpha}_{b-} g\right)^*(\tau) d\tau
= \int_a^b \left(\mathcal{I}^{\alpha}_{b-} f\right)(\tau) \cdot g^*(\tau)d\tau.
\end{equation}
\end{theorem}

\begin{proof}
From \eqref{eq:14} and \eqref{eq:11} it follows that
\begin{equation*}
\begin{split}
\int_a^b f(\tau) \left(\mathcal{I}^{\alpha}_{b-} g\right)^*(\tau) d\tau
&= \int_a^b f(\tau) \cdot \left(\mathcal{I}^{\alpha}_{a+} g^*\right)(\tau) d\tau\\
&= \int_a^b \left(\mathcal{I}^{\alpha}_{b-} f\right)(\tau) \cdot g^*(\tau) d\tau
\end{split}
\end{equation*}
and the result is proved.
\end{proof}

We now give a proper formula of integration by parts for the FCoV
that involves left fractional derivatives only.

\begin{theorem}[Integration by parts involving left fractional derivatives only]
\label{thm:03}
If $f \in \mathcal{I}^{\alpha}_{a+} \left(L^p(a,b)\right)$
and $g^* \in \mathcal{I}^{\alpha}_{b-} \left(L^q(a,b)\right)$ 
with $1/p + 1/q \leq 1 + \alpha$, then the following equality holds:
\begin{equation}
\label{eq:20}
\int_a^b f(t) \cdot \left(\mathcal{D}^{\alpha}_{a+} g\right)^*(t) dt
= -\int_a^b \left(\mathcal{D}^{\alpha}_{a+} f\right)(t) \cdot g^*(t)dt.
\end{equation}
\end{theorem}

\begin{proof}
From \eqref{eq:17} and \eqref{eq:12} we can write that
\begin{equation*}
\begin{split}
\int_a^b f(t) \left(\mathcal{D}^{\alpha}_{a+} g\right)^*(t) dt
&= -\int_a^b f(t) \cdot \left(\mathcal{D}^{\alpha}_{b-} g^*\right)(t) dt\\
&= -\int_a^b \left(\mathcal{D}^{\alpha}_{a+} f\right)(t) \cdot g^*(t) dt,
\end{split}
\end{equation*}
which proves the intended relation.
\end{proof}

In classical mechanics it makes sense to consider causal operators.
For this reason, we shall develop our new FCoV based on left operators
only. For completeness, however, we also provide here a similar relation 
to \eqref{eq:20} involving right fractional derivatives only.

\begin{theorem}[Integration by parts involving right fractional derivatives only]
\label{thm:04}
If $f \in \mathcal{I}^{\alpha}_{b-} \left(L^p(a,b)\right)$
and $g^* \in \mathcal{I}^{\alpha}_{a+} \left(L^q(a,b)\right)$ 
with $1/p + 1/q \leq 1 + \alpha$, then the following equality holds:
\begin{equation}
\label{eq:21}
\int_a^b f(t) \cdot \left(\mathcal{D}^{\alpha}_{b-} g\right)^*(t) dt
= - \int_a^b \left(\mathcal{D}^{\alpha}_{b-} f\right)(t) \cdot g^*(t)dt.
\end{equation}
\end{theorem}

\begin{proof}
Using Lemma~\ref{lemma:02} and Proposition~\ref{prop:2.5}, we obtain that
\begin{equation*}
\begin{split}
\int_a^b f(t) \cdot \left(\mathcal{D}^{\alpha}_{b-} g\right)^*(t) dt
&\stackrel{\eqref{eq:16}}{=}
-\int_a^b f(t) \cdot \left(\mathcal{D}^{\alpha}_{a+} g^*\right)(t) dt\\
&\stackrel{\eqref{eq:12}}{=}
-\int_a^b \left(\mathcal{D}^{\alpha}_{b-} f\right)(t) \cdot g^*(t)dt.
\end{split}
\end{equation*}
The result is proved.
\end{proof}

% -----------

\subsection{From fractional to classical derivatives}
\label{sec:fromFDtoCD}

Before formulating our fundamental problem of the fractional
calculus of variations and proving its Euler--Lagrange necessary
optimality condition, we prove here several auxiliary results
that relate expressions involving fractional derivatives
into expressions involving classical derivatives only.
The results are here formulated and proved for left derivatives
but similar formulas for right derivatives can also be obtained.

\begin{proposition}
\label{prop:03}
Let $\alpha_1, \alpha_2 \in [0,1]$ with $\alpha_1 + \alpha_2 = 1$. 
If $h \in C^1([a, b];\mathbb{R})$, then
\begin{equation}
\label{eq:22}
\left(\mathcal{D}^{\alpha_1}_{a+}\left(\mathcal{D}^{\alpha_2}_{a+} 
h\right)\right)(t)=h'(t)
\end{equation}
for all $t \in [a,b]$.
\end{proposition}

\begin{proof}
Let $\psi = \mathcal{D}^{\alpha_2}_{a+} h$. Then,
\begin{equation*}
\begin{split}
\left(\mathcal{D}^{\alpha_1}_{a+}\psi\right)(t)
&\stackrel{\eqref{eq:05b}}{=} 
\frac{d}{dt} \left[\left(\mathcal{I}^{1-\alpha_1}_{a+} \psi\right)(t)\right]
= \frac{d}{dt} \left[\left(\mathcal{I}^{\alpha_2}_{a+} 
\mathcal{D}^{\alpha_2}_{a+} h\right)(t)\right]\\
&\stackrel{\eqref{eq:08}}{=} \frac{d}{dt} \left[h(t)\right] = h'(t).
\end{split}
\end{equation*}
The proof is complete.
\end{proof}

\begin{theorem}
\label{thm:05}
Let $\alpha_1, \alpha_2 \in [0,1]$ with $\alpha_1 + \alpha_2 = 1$.
If $f \in C^1([a, b];\mathbb{R})$, then
\begin{equation}
\label{eq:23}
\int_a^b \left(\mathcal{D}^{\alpha_1}_{a+} f\right)(t) 
\cdot \left(\mathcal{D}^{\alpha_2}_{a+} g\right)^*(t) dt
= -\int_a^b f'(t) \cdot g^*(t)dt.
\end{equation}
\end{theorem}

\begin{proof}
Using Theorem~\ref{thm:03} and Proposition~\ref{prop:03} we obtain:
\begin{equation}
\begin{split}
\int_a^b \left(\mathcal{D}^{\alpha_1}_{a+} f\right)(t) 
\cdot \left(\mathcal{D}^{\alpha_2}_{a+} g\right)^*(t) dt
&\stackrel{\eqref{eq:20}}{=} 
-\int_a^b \left(\mathcal{D}^{\alpha_2}_{a+} 
\left(\mathcal{D}^{\alpha_1}_{a+} f\right)\right)(t) \cdot g^*(t)dt\\
&\stackrel{\eqref{eq:22}}{=} 
-\int_a^b f'(t) \cdot g^*(t)dt.
\end{split}
\end{equation}
The result is proved.
\end{proof}

\begin{remark}
\label{rem:03}
As a consequence of Proposition~\ref{prop:04}, 
the left-hand side of \eqref{eq:23} can be written as
\begin{equation*}
\int_a^b \left(\mathcal{D}^{\alpha_1}_{a+} f\right)^*(t) 
\cdot \left(\mathcal{D}^{\alpha_2}_{a+} g\right)(t) dt
\end{equation*}
while the right-hand side of \eqref{eq:23} can be written as
\begin{equation*}
-\int_a^b \left(f'\right)^*(t) \cdot g(t)dt.
\end{equation*}
\end{remark}

\begin{corollary}
\label{cor:07}
Let $\alpha_1, \alpha_2 \in [0,1]$ with $\alpha_1 + \alpha_2 = 1$.
If $f \in C^1([a, b];\mathbb{R})$ and $h \in L^1([a, b],\mathbb{R})$, then
\begin{equation}
\label{eq:24}
\int_a^b \left(\mathcal{D}^{\alpha_1}_{a+} f\right)(t) 
\cdot \left(\mathcal{D}^{\alpha_2}_{b-} h\right)(t) dt
= \int_a^b f'(t) \cdot h(t)dt.
\end{equation}
\end{corollary}

\begin{proof}
Using \eqref{eq:23} with $g = h^* \Leftrightarrow h = g^*$ we obtain that
\begin{equation}
\label{eq:24b}
\int_a^b \left(\mathcal{D}^{\alpha_1}_{a+} f\right)(t) 
\cdot \left(\mathcal{D}^{\alpha_2}_{a+} h^*\right)^*(t) dt
= -\int_a^b f'(t) \cdot h(t)dt.
\end{equation}
From \eqref{eq:16b} and Proposition~\ref{prop:01} we know that
\begin{equation*}
\left(\mathcal{D}^{\alpha_2}_{a+} h^*\right)^*
= \mathcal{D}^{\alpha_2}_{b-} (-h)
= - \mathcal{D}^{\alpha_2}_{b-} h
\end{equation*}
and \eqref{eq:24b} can be rewritten as
\begin{equation*}
-\int_a^b \left(\mathcal{D}^{\alpha_1}_{a+} f\right)(t) 
\cdot \left(\mathcal{D}^{\alpha_2}_{b-} h\right)(t) dt
= -\int_a^b f'(t) \cdot h(t)dt,
\end{equation*}
which proves the result.
\end{proof}

\begin{proposition}
\label{prop:05}
Let $\alpha_1, \alpha_2 \in [0,1]$ with $\alpha_1 + \alpha_2 = 1$.
If $f^* \in C^1([a, b],\mathbb{R})$, then
\begin{equation}
\label{eq:prop:05}
\mathcal{D}^{\alpha_1}_{a+}\left(\mathcal{D}^{\alpha_2}_{b-}f\right)^* = -\left(f^*\right)'.
\end{equation}
\end{proposition}

\begin{proof}
Let $\psi = \left(\mathcal{D}^{\alpha_2}_{b-}f\right)^*$. Then,
\begin{equation*}
\begin{split}
\left(\mathcal{D}^{\alpha_1}_{a+}\left(\mathcal{D}^{\alpha_2}_{b-}f\right)^*\right)(t) 
&= \left(\mathcal{D}^{\alpha_1}_{a+}\psi\right)(t)\\
&\stackrel{\eqref{eq:05b}}{=}
\frac{d}{dt} \left[\left(\mathcal{I}^{\alpha_2}_{a+} \psi\right)(t)\right]\\
&= \frac{d}{dt} \left[\left(\mathcal{I}^{\alpha_2}_{a+} 
\left(\mathcal{D}^{\alpha_2}_{b-}f\right)^*\right)(t)\right]\\
&\stackrel{\eqref{eq:16}}{=} 
\frac{d}{dt} \left[\left(\mathcal{I}^{\alpha_2}_{a+} 
\left(-\mathcal{D}^{\alpha_2}_{a+} f^*\right)\right)(t)\right]\\
&\stackrel{\eqref{eq:08}}{=} 
- \frac{d}{dt} \left[f^*(t)\right] = - \left(f^*\right)'(t),
\end{split}
\end{equation*}
and the proof is complete.
\end{proof}

% -----------

\subsection{A new FCoV involving left fractional operators only}
\label{sec:newFCoV}

Here we propose a different approach to the FCoV so that
both variational problems and their necessary 
optimality conditions involve left fractional 
derivatives only. This contrasts completely 
with the results found in the literature.

Let $0 < \alpha < 1$ and let us use the standard notation
of mechanics for the derivative with respect to time:
$\dot{x}(t) = x'(t)$. Consider the following problem: 
find a function $x \in C^1([a, b]; \mathbb{R})$
that gives an extremum (minimum or maximum)
to the integral functional $\mathcal{J}$,
\begin{equation}
\label{eq:P:Funct}
\mathcal{J}[x] 
= \int_{a}^{b} L\left(t,x(t),\dot{x}(t),\left(\mathcal{D}^{\alpha_1}_{a+}x\right)(t),
\left(\mathcal{D}^{\alpha_2}_{a+}x\right)^*(t)\right) dt,
\end{equation}
when subject to given boundary conditions
\begin{equation}
\label{eq:P:bc}
x(a) = x_a, \quad x(b)=x_b.
\end{equation}
We assume that the Lagrangian $L \in C^2([a, b] \times
\mathbb{R}^4;\mathbb{R})$ and that $\partial_4 L$ and $\partial_5 L$
(the partial derivatives of $L(\cdot,\cdot,\cdot,\cdot,\cdot)$ 
with respect to its 4th and 5th arguments, respectively) have 
continuous Riemann--Liouville fractional derivatives. 

\begin{definition}
A function $x \in C^1([a, b]; \mathbb{R})$ that satisfies the given boundary
conditions \eqref{eq:P:bc} is said to be an admissible trajectory for problem 
\eqref{eq:P:Funct}--\eqref{eq:P:bc}.
\end{definition}

Trivially, for the particular case when the Lagrangian $L$ 
only depends on $t$, $x$, and $\dot{x}$, then problem
\eqref{eq:P:Funct}--\eqref{eq:P:bc} reduces to the classical
fundamental problem of the calculus of variations \cite{MR2014219}. 
We now prove an extension of the classical 
Euler--Lagrange equation: when $L = L(t, x, \dot{x})$
in \eqref{eq:P:Funct}, then our necessary optimality 
condition \eqref{eq:EL} reduces to the classical 
Euler--Lagrange  equation:
\begin{equation*}
\partial_2 L  - \frac{d}{dt}\left(\partial_3 L\right) = 0.
\end{equation*}

\begin{theorem}[The fractional Euler--Lagrange equation]
\label{thm:ELeq}
If $x$ is an extremizer (minimizer or maximizer) 
to problem \eqref{eq:P:Funct}--\eqref{eq:P:bc}, 
then $x$ satisfies the Euler--Lagrange equation
\begin{equation}
\label{eq:EL}
\partial_2 L  
- \frac{d}{dt}\left(\partial_3 L\right)
- \left(\mathcal{D}^{\alpha_1}_{a+} \left(\partial_4 L\right)^* \right)^*
- \left(\mathcal{D}^{\alpha_2}_{a+} \left(\partial_5 L\right)\right)^* = 0.
\end{equation}
\end{theorem}

\begin{proof}
Suppose that $x$ is a solution of \eqref{eq:P:Funct}--\eqref{eq:P:bc}. 
Note that $\hat{x}_\epsilon(t) = x(t) + \epsilon h(t)$
is admissible for our problem for any $h \in C^1([a, b]; \mathbb{R})$ 
with $h(a) = h(b) = 0$ and for any $\epsilon \in \mathbb{R}$. 
Clearly, for $\epsilon = 0$ one gets the solution of the problem:
$\hat{x}_0 = x$. Let us define the real function 
$J : \mathbb{R}\rightarrow\mathbb{R}$ by
\begin{equation*}
J(\epsilon) =  \mathcal{J}[\hat{x}_\epsilon] 
= \int_{a}^{b} L\left(t,\hat{x}_\epsilon(t),\dot{\hat{x}}_\epsilon(t),
\left(\mathcal{D}^{\alpha_1}_{a+}\hat{x}_\epsilon\right)(t),
\left(\mathcal{D}^{\alpha_2}_{a+}\hat{x}_\epsilon\right)^*(t)\right) dt.
\end{equation*}
Since $\epsilon = 0$ is an extremizer for function $J$, the classical
Fermat's theorem asserts that a necessary optimality condition 
for our problem is given by $J'(0) = 0$. Using the linearity
of the involved operators, a direct computation tells us that
\begin{equation}
\label{eq:*:proof:ELeq}
\begin{split}
0 &= J'(0) = \left.\frac{d}{d\epsilon}J(\epsilon)\right|_{\epsilon = 0}\\
&= \int_{a}^{b} \left[\partial_2 L \cdot h(t) + \partial_3 L \cdot \dot{h}(t)
+ \partial_4 L \cdot \left(\mathcal{D}^{\alpha_1}_{a+}h\right)(t)
+ \partial_5 L \cdot \left(\mathcal{D}^{\alpha_2}_{a+}h\right)^*(t)\right] dt,
\end{split}
\end{equation}
where we omit the arguments 
$\left(t,x(t),\dot{x}(t),\left(\mathcal{D}^{\alpha_1}_{a+}x\right)(t),
\left(\mathcal{D}^{\alpha_2}_{a+}x\right)^*(t)\right)$
of $\partial_i L$, for $i = 2, \ldots, 5$.
Using the classical formula of integration by parts,
and having in mind that $h(a) = h(b) = 0$, it follows that
\begin{equation}
\label{part:01}
\int_{a}^{b} \partial_3 L \cdot \dot{h}(t) \, dt
= - \int_{a}^{b} \frac{d}{dt}\left(\partial_3 L\right) \cdot h(t) \, dt.
\end{equation}
On the other hand, our Proposition~\ref{prop:01} 
and Theorem~\ref{thm:03} allow us to write that
\begin{equation}
\label{part:02}
\begin{split}
\int_{a}^{b} \partial_4 L \cdot \left(\mathcal{D}^{\alpha_1}_{a+}h\right)(t) \, dt
&\stackrel{\eqref{eq:13b}}{=} 
\int_{a}^{b} \left(\left(\partial_4 L\right)^*\right)^* 
\cdot \left(\mathcal{D}^{\alpha_1}_{a+}h\right)(t) \, dt\\
&\stackrel{\eqref{eq:20}}{=}
- \int_{a}^{b} \left(\mathcal{D}^{\alpha_1}_{a+} \left(\partial_4 L\right)^* \right)^*
\cdot h(t) \, dt,
\end{split}
\end{equation}
while from Theorem~\ref{thm:03} and Proposition~\ref{prop:04} we obtain that
\begin{equation}
\label{part:03}
\begin{split}
\int_{a}^{b} \partial_5 L \cdot \left(\mathcal{D}^{\alpha_2}_{a+}h\right)^*(t) \, dt
&\stackrel{\eqref{eq:20}}{=}
-\int_{a}^{b} \mathcal{D}^{\alpha_2}_{a+} \left(\partial_5 L\right) \cdot h^*(t) \, dt\\
&\stackrel{\eqref{eq:prop:04}}{=}
-\int_{a}^{b} \left(\mathcal{D}^{\alpha_2}_{a+} \left(\partial_5 L\right)\right)^*
\cdot h(t) \, dt.
\end{split}
\end{equation}
Substituting \eqref{part:01}, \eqref{part:02} and \eqref{part:03} 
into \eqref{eq:*:proof:ELeq}, we conclude that
\begin{equation}
\label{eq:EL:with:h:factored}
\int_{a}^{b} \left[\partial_2 L  
- \frac{d}{dt}\left(\partial_3 L\right)
- \left(\mathcal{D}^{\alpha_1}_{a+} \left(\partial_4 L\right)^* \right)^*
- \left(\mathcal{D}^{\alpha_2}_{a+} \left(\partial_5 L\right)\right)^*\right] 
\cdot h(t) \, dt = 0.
\end{equation}
We obtain the intended necessary optimality condition \eqref{eq:EL}
applying the fundamental lemma of the calculus of variations 
to \eqref{eq:EL:with:h:factored}.
\end{proof}

As an application of our new proposed calculus of variations,
in Section~\ref{sec:ex} we give the first example of the literature
of a fractional mechanical Lagrangian $L$ involving left fractional 
derivatives only, for which the respective Euler--Lagrange equation 
coincides with the motion of a dissipative/damped system.

% -------------------------------

\section{A nonconservative/dissipative equation of motion}
\label{sec:ex}

Let $\alpha_1, \alpha_2 \in [0,1]$ with $\alpha_1 + \alpha_2 = 1$ 
and the Lagrangian $L$ be given by
\begin{equation}
\label{eq:L}
L\left(t,x,\dot{x},\mathcal{D}^{\alpha_1}_{a+}x,
\left(\mathcal{D}^{\alpha_2}_{a+}x\right)^*\right)
= \frac{1}{2} m \dot{x}^2 - U(x) 
+ \frac{1}{2} c \left(\mathcal{D}^{\alpha_1}_{a+}x\right)
\cdot \left(\mathcal{D}^{\alpha_2}_{a+}x\right)^*.
\end{equation}
In this case, one has
\begin{equation*}
\begin{split}
\partial_2 L &= -U'(x),\\
\partial_3 L &= m \dot{x},\\
\partial_4 L &= \frac{1}{2} c \left(\mathcal{D}^{\alpha_2}_{a+}x\right)^*
\Rightarrow \left(\partial_4 L\right)^* 
= \frac{1}{2} c \left(\mathcal{D}^{\alpha_2}_{a+}x\right),\\  
\partial_5 L &= \frac{1}{2} c \left(\mathcal{D}^{\alpha_1}_{a+}x\right),
\end{split}
\end{equation*}
and thus the Euler--Lagrange equation \eqref{eq:EL} takes the form
\begin{equation}
\label{ELF:ex}
-U'(x) - \frac{d}{dt} \left(m \dot{x}\right)
- \frac{1}{2} c \left(\mathcal{D}^{\alpha_1}_{a+}
\left(\mathcal{D}^{\alpha_2}_{a+}x\right)\right)^*
- \frac{1}{2} c \left(\mathcal{D}^{\alpha_2}_{a+}
\left(\mathcal{D}^{\alpha_1}_{a+}x\right)\right)^*
=0.
\end{equation}
It follows from Proposition~\ref{prop:03} that \eqref{ELF:ex} is equivalent to
\begin{equation}
\label{ELF:ex:de}
U'(x) + c (\dot{x})^* + m \ddot{x} = 0,
\end{equation}
that is, we have obtained an equation of motion \eqref{ELF:ex:de}
with a nonconservative Rayleigh term \cite{MR4261844}. 

% -------------------------------

\section{Conclusion}
\label{sec:conc}

In 1996, Riewe has proved that a calculus of variations can be formulated 
to include derivatives of fractional (non-integer) order \cite{CD:Riewe:1996}. 
By doing that, Riewe has shown possible to define Lagrangians, involving both 
left and right fractional derivatives, that lead directly to equations of motion
with nonconservative forces, such as friction, circumventing
Bauer's corollary that ``The equations of motion of a dissipative
linear dynamical system with constant coefficients are not given
by a variational principle'' \cite{Bauer:1931}. Here we continued the development
of the fractional-derivative calculus of variations by providing
a completely new perspective to the subject. Our main contributions are:
(i) new formulas of integration by parts
that involve left fractional operators only; 
(ii) a new fractional calculus of variations
where both the Lagrangian and respective Euler--Lagrange equation
involve left fractional derivatives only;
(iii) a new example of a Lagrangian for which the respective
Euler--Lagrange equation coincide with the equation 
of motion of a dissipative system.
Such contributions are radically different from other results
found in the vast literature on the subject. Indeed, to the best
of our knowledge, (i) all available formulas of fractional integration 
by parts involve both left and right fractional operators;
(ii)~all available forms of the fractional calculus of variations
for which the Lagrangian involves left fractional derivatives
result in Euler--Lagrange equations
involving a right fractional derivative; 
(iii) all available examples of fractional Lagrangians giving rise 
to equations of motion of a dissipative system involve both
left and right fractional derivatives.

In this paper we have restricted ourselves to ideas and to the central result of
any calculus of variations: the celebrated Euler--Lagrange equation, which is a first-order
necessary optimality condition. Of course our results can be extended in many different ways
and directions. We trust that the new perspective introduced here marks the beginning 
of a fruitful road for fractional mechanics, the fractional calculus of variations 
and fractional optimal control. 

% -------------------------------

\section*{Statements and Declarations}

\subsection*{Funding}

This research was supported by The Portuguese
Foundation for Science and Technology (FCT) and the 
Center for Research and Development in Mathematics 
and Applications (CIDMA), under Grant Agreement 
No UIDB/04106/2020 
(\url{https://doi.org/10.54499/UIDB/04106/2020}).

% --------------------------------------------------------

\subsection*{Availability of data and materials}

Not applicable.

% --------------------------------------------------------

\subsection*{Competing interests} 

The author has no competing interests to declare.

% --------------------------------------------------------

% -------------------------------


\begin{thebibliography}{xx}
	
\bibitem{MR3443073} 
R. Almeida, S. Pooseh\ and\ D. F. M. Torres, 
{\it Computational methods in the fractional calculus of variations}, 
Imperial College Press, London, 2015. 
\url{https://doi.org/10.1142/p991}
	
\bibitem{MR3822307} 
R. Almeida, D. Tavares\ and\ D. F. M. Torres, 
{\it The variable-order fractional calculus of variations}, 
SpringerBriefs in Applied Sciences and Technology, Springer, Cham, 2019. 
\url{https://doi.org/10.1007/978-3-319-94006-9}
{\tt arXiv:1805.00720}
	
\bibitem{MR3888408} 
R. Almeida\ and\ D. F. M. Torres, 
A survey on fractional variational calculus, 
in {\it Handbook of fractional calculus with applications. Vol. 1}, 
347--360, De Gruyter, Berlin, 2019.
\url{https://doi.org/10.1515/9783110571622-014}
{\tt arXiv:1806.05092}

\bibitem{MR4550832}
M. Al-Refai\ and\ A. Fernandez,  
Generalising the fractional calculus with Sonine kernels via conjugations, 
J. Comput. Appl. Math. 427 (2023), Paper No.~115159, 18~pp.
	
\bibitem{Bauer:1931} 
P. S. Bauer, 
Dissipative Dynamical Systems. I.
Proc. Natl. Acad. Sci. {\bf 17} (1931), no.~5, 311--314.
\url{https://doi.org/10.1073/pnas.17.5.311}

\bibitem{MR4261844}
A. M. Bersani\ and\ P. Caressa, 
Lagrangian descriptions of dissipative systems: a review, 
Math. Mech. Solids {\bf 26} (2021), no.~6, 785--803. 
\url{https://doi.org/10.1177/1081286520971834}
	
\bibitem{MyID:307} 
M. C. Caputo\ and\ D. F. M. Torres, 
Duality for the left and right fractional derivatives,
Signal Process. 107 (2015), 265--271.
\url{https://doi.org/10.1016/j.sigpro.2014.09.026}
% {\tt arXiv:1409.5319}

\bibitem{MR3390552}
M. J. Lazo\ and\ C. E. Krumreich, 
The action principle for dissipative systems, 
J. Math. Phys. {\bf 55} (2014), no.~12, 122902, 11~pp.
\url{https://doi.org/10.1063/1.4903991} 

\bibitem{MR4085516}
D. J. N. Limebeer, S. Ober-Bl\"{o}baum\ and\ F. H. Farshi, 
Variational integrators for dissipative systems, 
IEEE Trans. Automat. Control {\bf 65} (2020), no.~4, 1381--1396.
\url{https://doi.org/10.1109/tac.2020.2965059} 
	
\bibitem{MR3331286} 
A. B. Malinowska, T. Odzijewicz\ and\ D. F. M. Torres, 
{\it Advanced methods in the fractional calculus of variations}, 
SpringerBriefs in Applied Sciences and Technology, Springer, Cham, 2015. 
\url{https://doi.org/10.1007/978-3-319-14756-7}
	
\bibitem{MR2984893} 
A. B. Malinowska\ and\ D. F. M. Torres, 
{\it Introduction to the fractional calculus of variations}, 
Imperial College Press, London, 2012. 
\url{https://doi.org/10.1142/p871}
	
\bibitem{MR3221831} 
T. Odzijewicz\ and\ D. F. M. Torres, 
The generalized fractional calculus of variations, 
Southeast Asian Bull. Math. {\bf 38} (2014), no.~1, 93--117. 
{\tt arXiv:1401.7291}

\bibitem{MR2014219}
P. Pedregal, 
{\it Introduction to optimization}, 
Texts in Applied Mathematics, 46, Springer-Verlag, New York, 2004. 
\url{https://doi.org/10.1007/b97412}

\bibitem{CD:Riewe:1996}  
F. Riewe,
Nonconservative Lagrangian and Hamiltonian mechanics,
Phys. Rev. E (3) {\bf 53} (2) (1996) 1890--1899.
\url{https://doi.org/10.1103/PhysRevE.53.1890}
	
\bibitem{CD:Riewe:1997}  
F. Riewe,
Mechanics with fractional derivatives,
Phys. Rev. E (3) {\bf 55} (3) (1997) 3581--3592.
\url{https://doi.org/10.1103/PhysRevE.55.3581}
	
\bibitem{MR1347689}  
S. G. Samko, A. A. Kilbas\ and\ O. I. Marichev, 
{\it Fractional integrals and derivatives}, 
Gordon and Breach Science Publishers, Yverdon, 1993. 
	
\end{thebibliography}
\end{document}